\documentclass[12pt]{amsart}

\usepackage{amsfonts}
\usepackage{amsmath}
\usepackage{amssymb}
\usepackage{amsthm}
\usepackage{enumerate} 
\usepackage{hyperref}
\usepackage[margin=1.2in]{geometry}
\usepackage{graphicx} 
\usepackage{todonotes}
\usepackage{xcolor}

\theoremstyle{plain}
\newtheorem{theorem}{Theorem}[section]
\newtheorem{lemma}[theorem]{Lemma}
\newtheorem{proposition}[theorem]{Proposition}
\newtheorem{corollary}[theorem]{Corollary}

\theoremstyle{definition}
\newtheorem{definition}[theorem]{Definition}

\newtheorem{remark}[theorem]{Remark}

\newtheorem{example}[theorem]{Example}
\newtheorem{examples}[theorem]{Examples}

\numberwithin{equation}{section}

\newcommand\N{\mathbb{N}}
\newcommand\Z{\mathbb{Z}}
\newcommand\R{\mathbb{R}}

\newcommand\C{\mathbb{C}}

\newcommand\D{\mathcal{D}}

\renewcommand\S{\mathcal{S}}

\newcommand\B{\mathcal{B}}

\newcommand\ev[2]{\langle#1,#2\rangle}

\DeclareMathOperator\id{id}

\begin{document}

\title[Sequence space representations for $\D_E$ and $\D'_E$]{Sequence space representations for translation-modulation invariant function and distribution spaces}

\author[A. Debrouwere]{Andreas Debrouwere}
\address{Department of Mathematics and Data Science \\ Vrije Universiteit Brussel, Belgium\\ Pleinlaan 2 \\ 1050 Brussels \\ Belgium \newline
\indent \emph{ORCID:} 0000-0003-4416-5758
}
\email{andreas.debrouwere@vub.be}

\author[L. Neyt]{Lenny Neyt}
\thanks{L. Neyt gratefully acknowledges support by FWO-Vlaanderen through the postdoctoral grant 12ZG921N}
\address{Department of Mathematics: Analysis, Logic and Discrete Mathematics\\ Ghent University\\ Krijgslaan 281\\ 9000 Gent\\ Belgium \newline
\indent \emph{ORCID:} 0000-0001-8116-1487}
\email{lenny.neyt@UGent.be}


\begin{abstract}
We provide sequence space representations for  the test function space $\D_{E}$ and the distribution space $\D^{\prime}_{E}$ associated to a Banach space $E$ belonging to a broad class of translation-modulation invariant Banach spaces of distributions. The spaces $\D_{E}$ and $\D^{\prime}_{E}$ generalize the classical Schwartz spaces $\D_{L^p}$ and $\D^{\prime}_{L^p}$, respectively.
Our proof is based on Gabor frame characterizations of $\D_{E}$ and $\D^{\prime}_{E}$, which are also established here and are of independent interest.  We recover in a unified way some known sequence space representations 
as well as obtain several new ones.  
\end{abstract}
\maketitle

\section{Introduction}
The representation of test function and distribution spaces by sequence spaces is a central topic in functional analysis that goes back to the pioneering work of Valdivia and Vogt  \cite{V-TopicLCS, V-SeqSpRepSpTestFuncDist} and is closely connected to the isomorphic classification of such spaces. 

With regard to the convolution of general distributions, Schwartz \cite{Schwartz} introduced the test function spaces 
$$
\D_{L^p}  = \{ f \in C^\infty(\R^{n}) \mid f^{(\alpha)} \in L^p , ~ \forall \alpha \in \N^{n} \}, \qquad 1 \leq p \leq \infty,   
$$
$$
\dot{\mathcal{B}}  = \{ f \in C^\infty(\R^{n}) \mid \lim_{|x| \to \infty}f^{(\alpha)}(x) = 0, ~ \forall \alpha \in \N^{n} \}, \qquad 
$$
equipped with their natural Fr\'echet space topology, and the distribution spaces ($q$ denotes the conjugate index of $p$)
$$
\D'_{L^1} = (\dot{\mathcal{B}})', \qquad \D'_{L^p} = (\D_{L^q})', \qquad  1 < p \leq \infty, 
$$ 
endowed with their strong dual  topology. Following Schwartz, we write $\D_{L^\infty} = \mathcal{B}$ and  $\D'_{L^\infty} = \mathcal{B}'$ from now on. The space  $\dot{\mathcal{B}}'$ is defined as the closure of the space of compactly supported distributions in $\B'$. There has been a considerable interest in sequence space representations for these spaces. 
Valdivia  \cite{Valdivia} and Vogt  \cite{V-SeqSpRepSpTestFuncDist} independently showed that $\D_{L^p} \cong \ell^p \widehat{\otimes} s$, where 
\[ s = \{ c = (c_k)_{k \in \Z^n} \in \C^{\Z^n} \mid \sup_{k \in \Z^n}(1 + |k|)^{m} |c_{k}| < \infty, ~ \forall m \in \N \} \] 
denotes the Fr\'echet space of rapidly decreasing sequences. Vogt \cite{V-SeqSpRepSpTestFuncDist} also proved that  $\B \cong \ell^\infty \widehat{\otimes} s$ and  $\dot{\B} \cong c_0 \widehat{\otimes} s$. By standard duality arguments, these results imply that  $\D'_{L^{p}} \cong \ell^p \widehat{\otimes} s'$ and $\B' \cong \ell^\infty \widehat{\otimes} s'$, where $s'$ stands for the strong dual of $s$. More recently, Bargetz \cite{B-ComplVVTab}  showed that $\dot{\B}^{\prime}  \cong c_0 \widehat{\otimes} s'$. The original proofs of these isomorphisms are non-constructive as they are based on the Pe\l czy\'nski decomposition method  (in the version of Vogt \cite{V-SeqSpRepSpTestFuncDist}). We refer to \cite{B-D-N-SeqSpRepSmoothFuncDistWilsonBases,O-W} for explicit and constructive sequence space representations of these spaces.

The main goal of this article is to unify and generalize the above results by providing sequence space representations for  test function and distribution spaces defined via a broad class of translation-modulation invariant Banach spaces as introduced in \cite{D-P-P-V-TMIBUltraDist}.

A Banach space $E$ is said to be a \emph{translation-modulation invariant Banach space of (tempered) distributions} (TMIB) on $\R^{n}$ if $E$ satisfies the dense continuous inclusions  $\S(\R^{n}) \hookrightarrow E \hookrightarrow \S'(\R^{n})$, $E$ is translation and modulation invariant, and the operator norms of the translation and modulation operators on $E$ are polynomially bounded. Here, $\S(\R^{n})$ denotes the Fr\'echet space of rapidly decreasing smooth functions and $\S^{\prime}(\R^{n}) =(\S(\R^{n}))'$ the space of tempered distributions endowed with its strong dual topology \cite{Schwartz}. A Banach space $E$ is called a dual TMIB (DTMIB) if it is the strong dual of a TMIB. Classical examples of TMIB and DTMIB are the Lebesgue spaces $L^p$, $1 \leq p \leq \infty$, the space $C_{0}$ of continuous functions vanishing at infinity, the mixed-norm spaces $L^{p_1}(L^{p_2})$, $1 \leq p_1, p_2 < \infty$, the Sobolev spaces $\mathcal{F}L^p$, Wiener amalgam spaces, and the weighted variants of all these spaces.  Following  \cite{F-G-BSpIntegrGroupRepAtomicDecomp}, we call a Banach space $X$ a \emph{solid Banach function space}  (for short \emph{solid}) if $X \subset L^1_{\operatorname{loc}}(\R^n)$ with continuous inclusion and 
	\[ \forall f \in X, g \in L^1_{\operatorname{loc}}(\R^n) \, : \; |g| \leq |f| \mbox{ a.e.} \Rightarrow g \in X \mbox{ and } \|g\|_X \leq \| f \|_X . \]
Interestingly, TMIB and DTMIB are not necessarily solid, even if they consist of locally integrable functions.
A natural example of a non-solid TMIB consisting of locally integrable functions is given by $L^p \widehat{\otimes}_\pi L^p$, $1< p \leq 2$ \cite[Remark 3.10]{D-P-P-V-TMIBUltraDist}. We refer to \cite{D-P-P-V-TMIBUltraDist}  for a systematic study of TMIB and DTMIB.

Let $E$ be a TMIB or DTMIB on $\R^n$. We define 
 	\[ \D_{E} = \{ f \in \S^{\prime}(\R^{n}) \mid f^{(\alpha)} \in E , ~ \forall \alpha \in \N^{n} \} \]
	and
\[ \D^{\prime}_{E} = \{ f \in \S^{\prime}(\R^{n}) \mid f * \varphi \in E , ~ \forall \varphi \in \S(\R^{n}) \}, \]
and endow these spaces with their natural locally convex topologies (see Subsection \ref{subsect:testdist} for details). We remark that $\D_{E}$ was introduced in \cite{D-P-V}, while $\D'_{E}$ is considered here for the first time. The definition of $\D'_{E}$ is inspired by the fact that, if $E = E'_0$ is a DTMIB, it holds that 
 $(\D_{E_0})' = \D'_{E}$
as vector spaces, see  \cite[Theorem 3]{D-P-V}. For $E_0 = L^p$, $1 \leq p < \infty$, or $C_0$ this is a classical result of Schwartz  \cite[p.\ 201, Theor\'eme XXV]{Schwartz}. 

 In the main result of this paper (Theorem \ref{main}), we obtain sequence space representations for the spaces $\D_{E}$ and $\D'_{E}$. To this end,  we consider the discrete space associated to $E$ as recently introduced in \cite{D-P-GaborFrameCharGenModSp}
 (see Definition \ref{def-discrete}); this notion is heavily inspired by \cite[Definition 3.4]{F-G-BSpIntegrGroupRepAtomicDecomp}, where a discrete space is associated to each  translation invariant solid Banach function space. The proof of our main result is based on Vogt's variant of the Pe\l czy\'nski decomposition method  \cite{V-SeqSpRepSpTestFuncDist} and discrete characterizations of the spaces  $\D_{E}$ and $\D'_{E}$ in terms of Gabor frames \cite{G-FoundationsTimeFreqAnal}. The latter is the technical core of our work  and is of independent interest. We would like to point out that the  method of Vogt \cite[Theorem 3.1]{V-SeqSpRepSpTestFuncDist} for obtaining sequence space representations of $L^p$-weighted spaces of smooth functions does not seem to be applicable in our general situation, mainly due to the absence of solidity assumptions on $E$.

Our work delivers a new and unified proof of the sequence  space representations of the spaces $\D_{L^p}, \mathcal{B}, \dot{\mathcal{B}}$ and $\D'_{L^p}, \mathcal{B'}, \dot{\mathcal{B}}'$ (see Example \ref{examples-main}) but also gives several new ones. Let us present here two examples: 

\noindent $(i)$  For $1 \leq p_1, p_2 < \infty$ we have that
$$
\mathcal{D}_{L^{p_1} \widehat{\otimes}_\pi L^{p_2}} \cong 
\ell^{p_1} \widehat{\otimes}_\pi  \ell^{p_2} \widehat{\otimes}_\pi s.
$$ 
$(ii)$ For $1 \leq p \leq \infty$ and $r \in \R$ we define the Banach space
$$
\mathcal{F}L^p_{(1+|\, \cdot \,|)^r } =  \{ f \in \S'(\R^n) \, | \,  \ \| \mathcal{F}^{-1}(f)(\xi)(1+|\xi|)^r \|_{L^p} < \infty \}.
$$
Consider the projective and inductive type Sobolev spaces 
\begin{equation}
\label{sobolev-def}
\mathcal{F}L^p_{\infty} = \varprojlim_{m \in \N} \mathcal{F}L^p_{(1+|\, \cdot \,|)^m}, \qquad \mathcal{F}L^p_{- \infty} = \varinjlim_{m \in \N} \mathcal{F}L^p_{(1+|\, \cdot \,|)^{-m}}.
\end{equation}
Then,
\begin{equation}
	\label{sobolev-seq}
\mathcal{F}L^p_{ \infty}  \cong b^p  \widehat{\otimes} s, \qquad \mathcal{F}L^p_{ -\infty}  \cong b^p  \widehat{\otimes} s',
\end{equation}
where 
$$
b^p = \{ c =(c_j)_{j \in \Z^n} \, | \, \sum_{j \in \Z^n} c_j e^{2\pi i j \cdot x  } \in L^p(\R^n/ \Z^n) \}
$$
with norm $\| c \|_{b^p} = \left \|  \sum_{j \in \Z^n} c_j e^{2\pi i j \cdot x  }  \right \|_{L^p(\R^n/ \Z^n)}$. 

As an application of our results, we study in  Section \ref{sect-iso}  the isomorphic structure of the  spaces $\D_{L^{p_1} \widehat{\otimes}_\pi L^{p_2}}$. We show that these spaces are genuinely new and distinct in the following sense (see Proposition \ref{isotensor-2}):  For $1 < p_1, p_2 < \infty$ we have that 
\begin{itemize}
\item[$(i)$]  $\D_{L^{p_1} \widehat{\otimes}_\pi L^{p_2}}$ is not isomorphic to a complemented subspace of $\mathcal{D}_{L^{q_1}(L^{q_2})}$, $1 \leq q_1, q_2 < \infty$, $\dot{\mathcal{B}}$, or $\mathcal{B}$.
\item[$(ii)$] Let $1 < q_1, q_2 < \infty$.  Then, $\D_{L^{p_1} \widehat{\otimes}_\pi L^{p_2}} \cong \D_{L^{q_1} \widehat{\otimes}_\pi L^{q_2}}$ if and only if $\{p_1,p_2\} = \{q_1, q_2\}$.
\end{itemize}
By using the sequence space representations of the involved spaces and a functional analytic lemma from \cite{Diaz}, this will follow from some deep results about the isomorphic structure of the spaces $\ell^{p_1} \widehat{\otimes}_\pi  \ell^{p_2}$,  $\ell^{p_1}(\ell^{p_2})$,  $\ell^p$, and $c_0$ \cite{AK, CM,GL}.

This article is organized as follows. In the preliminary Sections \ref{sect-prelim} and \ref{sect-TMIB} we fix the notation and recall several results concerning TMIB and DTMIB \cite{D-P-P-V-TMIBUltraDist} and their associated discrete spaces from \cite{D-P-GaborFrameCharGenModSp}. We also formally introduce the spaces $\D_E$ and $\D'_E$ here. In  Section \ref{sect-stat} we state our main result, namely, the sequence space representations for  $\D_E$ and $\D'_E$, give some consequences of it about the linear topological structure of these spaces, and provide several examples. The Gabor frame characterizations of $\D_E$ and $\D'_E$, and the proof of the main result are given in Section \ref{sec:GaborFrameChar}. Finally, in Section \ref{sect-iso}, we discuss the  isomorphic structure of the  spaces $\D_{L^{p_1} \widehat{\otimes}_\pi L^{p_2}}$.

\section{Preliminaries} \label{sect-prelim}
Given an open set $U \subseteq \R^{n}$, we denote by $\D(U)$ the space of all smooth functions with compact support contained in $U$. We write $\S(\R^{n})$ for the Fr\'echet space of rapidly decreasing smooth functions on $\R^{n}$  and endow it with the following family of norms:
	\[ \|\varphi\|_{S^{N}} := \sup_{|\alpha| \leq N} \sup_{x \in \R^{n}} |\varphi^{(\alpha)}(x)| (1 + |x|)^{N} , \qquad N \in \N . \]
Its strong dual $\S^{\prime}(\R^{n})$ is the space of tempered distributions.  We denote the translation and modulation operators on $\S^{\prime}(\R^{n})$ by $T_{x} f(t) = f(t- x)$ and $M_{\xi} f(t) = e^{2 \pi i \xi \cdot t} f(t)$, $x, \xi \in \R^{n}$. We also set $\check{f}(t) = f(-t)$.  We fix the constants in the Fourier transform as follows
$$
\mathcal{F}(\varphi)(\xi) = \int_{\R^n} \varphi(x) e^{-2\pi i \xi \cdot x} dx, \qquad \varphi \in \S(\R^{n}).
$$
The Fourier transform is an isomorphism  from $\S(\R^{n})$ onto itself and extends by duality to an isomorphism from $\S'(\R^{n})$ onto itself. 

By a lattice $\Lambda$ in $\R^n$ we mean a discrete subgroup of $\R^{n}$ that spans the whole of $\R^{n}$. There is a unique invertible $n \times n$-matrix $A_{\Lambda}$ such that $\Lambda = A_{\Lambda} \Z^{n}$. A measurable function $w: \R^n \to (0,\infty)$ is called a polynomially bounded weight function on $\R^n$ if there are $C,\tau > 0$ such that
$$
w(x+y) \leq Cw(x)(1+|y|)^{\tau}, \qquad x,y \in \R^n.
$$
Let $X$ be a Banach space. For $1 \leq p \leq \infty$ we define the Banach space
$$
\ell^p_w(\Lambda;X) = \{ x = (x_\lambda)_{\lambda \in \Lambda} \in X^\Lambda \, | \, \|x\|_{\ell^p_w(\Lambda;X)} = \| \|x_\lambda \|_X w(\lambda)\|_{\ell^p} < \infty \}.
$$
Furthermore, we define  $c_{0,w}(\Lambda; X)$ as the closed subspace of $\ell^{\infty}_w(\Lambda;X)$ consisting of all $x\in\ell^{\infty}_w(\Lambda;X)$ satisfying the following property: For every $\varepsilon>0$ there is a finite subset $\Lambda^{(0)}$ of  $\Lambda$ such that $\sup_{\lambda\in\Lambda\backslash\Lambda^{(0)}}\|x_{\lambda}\|_X w(\lambda)\leq \varepsilon$. For $X = \C$ we simply write $\ell^p_w(\Lambda;\C) = \ell^p_w(\Lambda)$ and  $c_{0,w}(\Lambda; \C) = c_{0,w}(\Lambda)$. 

Next, we define
\[ s(\Lambda;X) = \varprojlim_{N \in \N} \ell^\infty_{(1+|\, \cdot \, |)^N}(\Lambda;X)  \quad \text{and} \quad s^{\prime}(\Lambda;X) = \varinjlim_{N \in \N} \ell^\infty_{(1+|\, \cdot \, |)^{-N}}(\Lambda;X). \]
Then, $s(\Lambda;X)$ is a Fr\'echet space, while $s'(\Lambda;X)$ is an $(LB)$-space (i.e. a countable inductive limit of Banach spaces). For $X = \C$ we again simply write $s(\Lambda;\C) = s(\Lambda)$ and $s'(\Lambda;\C) = s'(\Lambda)$. Note that, for any lattice $\Lambda$, $s(\Lambda)$ and $s'(\Lambda)$ are topologically isomorphic to $s = s(\Z^{n})$ and $s' = s'(\Z^{n})$, respectively. 
Finally,  we have the canonical isomorphisms
	\begin{equation}
	s(\Lambda;X) \cong s(\Lambda) \widehat{\otimes} X
	\label{tensor-1} 
	\end{equation}
and
\begin{equation}
	 s^{\prime}(\Lambda;X) \cong s^{\prime}(\Lambda) \widehat{\otimes} X,
	 \label{tensor-2} 
	\end{equation}
where the completion of the tensor products may be taken with respect to either the $\pi$- or the $\varepsilon$-topology, which amounts to the same in view of the nuclearity of $s(\Lambda)$ and $s'(\Lambda)$. The equality in \eqref{tensor-1} can be shown as follows
$$
s(\Lambda;X) = \varprojlim_{N \in \N} c_{0,(1+|\, \cdot \, |)^N}(\Lambda;X) = \varprojlim_{N \in \N} c_{0,(1+|\, \cdot \, |)^N}(\Lambda) \widehat{\otimes}_\varepsilon X =  s(\Lambda) \widehat{\otimes} X,
$$
while \eqref{tensor-2} is a consequence of \cite[Theorem 3.1(d)]{B-M-S}. We refer to \cite{G-ProdTensTopEspNucl, Ryan} for more information on completed tensor products.

\section{Translation-modulation invariant Banach spaces of distributions and their duals}\label{sect-TMIB}
\subsection{Definition and basic properties} We start with the following basic  definition from \cite{D-P-P-V-TMIBUltraDist}.
\begin{definition} \label{Def-1} A Banach space $E$ is called a \emph{translation-modulation invariant Banach space of distributions} (TMIB) on $\R^n$ if the following three conditions hold:
\begin{itemize}
\item[$(i)$] $E$ satisfies the dense continuous inclusions $\mathcal{S}(\R^n) \hookrightarrow E \hookrightarrow \mathcal{S}'(\R^n)$.
\item[$(ii)$] $T_x(E) \subseteq E$ and $M_\xi(E) \subseteq E$ for all $x, \xi \in \R^n$.
\end{itemize}
\begin{itemize}
\item[$(iii)$] There exist $\tau_j, C_j > 0$, $j=0,1$, such that
\begin{equation}
 \| T_{x} \|_{\mathcal{L}(E)} \leq C_0(1+|x|)^{\tau_0} \quad \mbox{and} \quad \| M_{\xi} \|_{\mathcal{L}(E)} \leq  C_1(1+|\xi|)^{\tau_1} ,
\label{twf}
\end{equation}
for $x, \xi \in \R^n$ fixed; here we used the fact that the mappings $T_x: E \rightarrow E$ and $M_\xi: E \rightarrow E$  are continuous, as follows from the closed graph theorem.
\end{itemize}
A Banach space $E$ is called a  \emph{dual translation-modulation invariant Banach space of distributions} (DTMIB) on $\R^n$ if it is the strong dual of a TMIB on $\R^n$.
\end{definition}

\begin{remark} 
Let $E$ be a DTMIB. Then, $E$ satisfies the continuous inclusions $\mathcal{S}(\R^n) \rightarrow E \rightarrow \mathcal{S}'(\R^n)$ and the conditions $(ii)$ and $(iii)$ from Definition \ref{Def-1}. However, in general, $E$ is not a TMIB as the inclusion $\mathcal{S}(\R^n) \rightarrow E$ need not be dense; consider, e.g., $E = L^\infty$. If $E$ is reflexive, then $E$ is in fact a TMIB  \cite[p.\ 827]{D-P-P-V-TMIBUltraDist}. 
\end{remark}

We now give some examples of TMIB and DTMIB; see also \cite[Section 3]{D-P-P-V-TMIBUltraDist} and \cite[Examples 3.3]{D-P-GaborFrameCharGenModSp}.
\begin{examples}\label{examples-TMIB}
$(i)$ A Banach space $E$ is called a \emph{solid TMIB (DTMIB)} on $\R^n$ if $E$ is both a TMIB (DTMIB) and a solid Banach function space (cf.\ the introduction). In such a case, $\| M_\xi e \|_E = \| e \|_E$ for all $e \in E$ and $\xi \in \R^n$.  The Lebesgue spaces $L^p = L^p(\R^n)$, $1 \leq p \leq \infty$, and the mixed norm spaces $L^{p_1}(L^{p_2})$, $1 \leq p_1,p_2 < \infty$, are  of this type.

$(ii)$ Let $E$ be a solid TMIB (DTMIB) on $\R^n$ and let $w$ be a polynomially bounded weight function on $\R^n$. We define the Banach space $E_w = \{ f \in  L^1_{\operatorname{loc}}(\R^n)  \, | \, f w \in E \}$ with  norm $\| f \|_{E_w} := \|  f  w\|_E$. Then, $E_w$ is again a solid TMIB (DTMIB). The weighted Lebesgue spaces $L^p_w = L^p_w(\R^n)$, $1 \leq p \leq \infty$, are of this type. 

$(iii)$ We define $C_{0,w} = C_{0,w}(\R^n)$ as the closed subspace of $L^{\infty}_w$ consisting of all $f\in C(\R^n)$ such that $\lim_{|x| \to \infty} f(x)w(x) = 0$. Then, $C_{0,w}$ is a TMIB.


$(iv)$ Given a TMIB (DTMIB) $E$, we define its associated Fourier space as the Banach space $\mathcal{F} E := \{ f \in  \mathcal{S}'(\R^n) \, | \, \mathcal{F}^{-1} f \in E \}$ with  norm $\| f \|_{\mathcal{F}E} := \| \mathcal{F}^{-1} f \|_E$. Then, $\mathcal{F}E$ is again a TMIB (DTMIB). The Sobolev spaces $\mathcal{F}L^p$, $1 \leq p \leq \infty$, are of this type. 

$(v)$ Let $E_j$ be a TMIB on $\R^{n_j}$ for $j = 1,2$. Let $\tau$ denote either $\pi$ or $\varepsilon$. Then, \cite[Theorem 3.6]{D-P-P-V-TMIBUltraDist} (together with \cite[Lemma 2.3]{F-P-P-ModSpAssocTensProdAmalgamSp} for $\tau = \pi$) yields that $E_1 \widehat{\otimes}_{\tau} E_2$ is a TMIB on $\R^{n_1+n_2}$. 
In particular, the spaces $L^{p_1}(\R^{n_1})  \widehat{\otimes}_{\tau} L^{p_2}(\R^{n_2})$ are TMIB on $\R^{n_1+n_2}$ of this type consisting of locally integrable functions. In \cite[Remark 3.10]{D-P-P-V-TMIBUltraDist} it is shown that $L^{p}(\R^{n})  \widehat{\otimes}_\pi L^{p}(\R^{n})$, $1 < p \leq 2$, is not solid.
\end{examples}

\subsection{Test function and distribution spaces associated to TMIB and DTMIB}\label{subsect:testdist} Let $E$ be a TMIB or DTMIB on $\R^n$. As in the introduction, we set (see \cite[Section 4.1]{D-P-V})
	\[ \D_{E} = \{ f \in \S^{\prime}(\R^{n}) \mid f^{(\alpha)} \in E , ~ \forall \alpha \in \N^{n} \} , \]
	and endow this space with the Fr\'echet space structure generated by the family of norms
	\[ \|f\|_{E,N} := \sup_{|\alpha| \leq N} \|f^{(\alpha)}\|_{E} , \qquad  N \in \N. \]  
\begin{examples}\label{examples-test}
$(i)$  Let $w$ be a polynomially bounded weight function on $\R^n$. The spaces $\D_{L^p_w}$ ($1 \leq p < \infty$), $\mathcal{B}_w = \D_{L^\infty_w}$, and $\dot{\mathcal{B}}_w = \D_{C_{0,w}}$ are weighted variants of the classical Schwartz spaces $\D_{L^p}$,  $\mathcal{B}$, and $\dot{\mathcal{B}}$ \cite{Schwartz} discussed in the introduction. 

$(ii)$  Let $E$ be a solid (D)TMIB. We define $\mathcal{F}E_\infty =   \varprojlim_{m \in \N}\mathcal{F} (E_{\left(1+ |\,\cdot\,|\right)^m})$. Then, $\mathcal{D}_{\mathcal{F}E} = \mathcal{F}E_\infty$ as locally convex spaces.
\end{examples}

Next, we recall from the introduction that
	\[ \D^{\prime}_{E} = \{ f \in \S^{\prime}(\R^{n}) \mid f * \varphi \in E , ~ \forall \varphi \in \S(\R^{n}) \}. \]
 By the closed graph theorem, we see that, for $f \in \D^{\prime}_{E}$ fixed, the linear mapping 
$
 \mathcal{S}(\R^{n}) \rightarrow E, \,  \varphi \mapsto f * \varphi
 $
  is continuous. Hence, we may  endow $\D^{\prime}_{E}$ with the topology induced by the embedding
	\[ \D^{\prime}_{E} \rightarrow \mathcal{L}_{b}(\S(\R^{n}), E), \, f \mapsto (\varphi \mapsto f * \varphi).  \]
	\begin{remark} \label{dual}
		Suppose that $E$ is a TMIB. As already stated in the introduction, we have that $(\D_E)' = \D^\prime_{E^\prime}$ as vector spaces \cite[Theorem 3]{D-P-V}. Moreover, if we endow $(\D_E)'$  with its strong dual topology, the
		inclusion mapping $(\D_E)' \rightarrow \D^\prime_{E^\prime}$ is continuous and the  spaces $(\D_E)'$ and $\D^\prime_{E^\prime}$ have the same bounded sets and null sequences  \cite[Corollary 4 and Corollary 5]{D-P-V}. Later on, as a consequence of our main result, we will show that $(\D_E)' = \D^\prime_{E^\prime}$ as locally convex spaces.
	\end{remark}

\begin{example}\label{examples-dual}
  Let $w$ be a polynomially bounded weight function on $\R^n$. The spaces $\D'_{L^p_w}$ ($1 \leq p < \infty$), $\mathcal{B}'_w = \D'_{L^\infty_w}$, and $\dot{\mathcal{B}}'_w = \D'_{C_{0,w}}$ are weighted variants of the classical Schwartz spaces $\D'_{L^p}$,  $\mathcal{B}'$, and $\dot{\mathcal{B}}'$ \cite{Schwartz} discussed in the introduction. Namely, by Remark \ref{dual}, we have that
\begin{align}
\label{dualexc}
&\D'_{L^1_w} = (\dot{\mathcal{B}}_{1/w})', \qquad \D'_{L^p_w} = (\D_{L^q_{1/w}})' \quad (1 < p < \infty), \\ \nonumber
& \mathcal{B}'_{w} = (\mathcal{D}_{L^1_{1/w}})', \qquad  \dot{\mathcal{B}}'_w = \overline{\mathcal{E}'(\R^n)}^{\mathcal{B}'_w} ,
\end{align}
as vector spaces, where $q$ denotes the conjugate index of $p$ and $\mathcal{E}'(\R^n)$ is the space of compactly supported distributions (see also\ \cite[p.\ 201, Theor\'eme XXV]{Schwartz}). 
\end{example}

Next, we give a description of the space $\D'_{\mathcal{F}E}$ when $E$ is solid. 

\begin{proposition}\label{fourier-solid}
Suppose that $E$ is solid. Define $\mathcal{F}E_{-\infty} = \varinjlim_{m \in \N}\mathcal{F} (E_{\left(1+ |\,\cdot\,|\right)^{-m}})$.  Then, $\mathcal{D}'_{\mathcal{F}E} =\mathcal{F}E_{-\infty}$ as vector spaces. 
\end{proposition} 
\begin{proof}
	We have that 
	$$
	 \mathcal{D}'_{\mathcal{F}E} =  \{ f \in \S^{\prime}(\R^{n}) \mid \mathcal{F}^{-1}f \cdot \varphi \in E , ~ \forall \varphi \in \S(\R^{n}) \}. 
	 $$
	 In particular, it holds that $\mathcal{F}^{-1}f \in  L^1_{\operatorname{loc}}(\R^n)$ for all $f \in \mathcal{D}'_{\mathcal{F}E}$.
	 
	The inclusion  $\mathcal{F}E_{-\infty} \subseteq \mathcal{D}'_{\mathcal{F}E}$  is  clear. We now show the reverse inclusion. 
	For $N \in \N$ we define $\S^{N}(\R^n)$ as the space consisting of all $\varphi \in C^N(\R^n)$ such that $\| \varphi \|_{\S^N} < \infty$.
	We claim that for all $f \in  \mathcal{D}'_{\mathcal{F}E}$ there is $N \in \N$ such that  $\mathcal{F}^{-1}f \cdot \varphi \in E$ for all $\varphi \in \S^{N}(\R^n)$. 
	Since the function $(1 + |x|^2)^{-N/2}$ belongs to $\S^{N}(\R^n)$, this would entail the result. We now show the claim.
	 Let $f \in  \mathcal{D}'_{\mathcal{F}E}$ be arbitrary. By the closed graph theorem, we find that the mapping
 $ \mathcal{S}(\R^{n}) \rightarrow E, \,  \varphi \mapsto \mathcal{F}^{-1}f \cdot \varphi$ is continuous. Hence, there are $M \in \N$ and $C > 0$ such that
	\begin{equation}
		\label{norm-cont}
	\| \mathcal{F}^{-1}f \cdot \varphi \|_{E} \leq C \| \varphi\|_{\S^M}, \qquad \varphi \in \mathcal{S}(\R^n).
	\end{equation}
Note that $\S(\R^n)$ is dense in $\S^{M+1}(\R^n)$ with respect to the norm $\| \, \cdot \|_M$. Let $\varphi \in \S^{M+1}(\R^n)$ be arbitrary. Choose a sequence $(\varphi_j)_{j \in \N} \subset \S(\R^n)$ such that $\| \varphi -\varphi_j\|_{\S^M} \to 0$ as $j \to \infty$. Then, $\mathcal{F}^{-1}f \cdot \varphi_j \to \mathcal{F}^{-1}f \cdot \varphi$ in $L^1_{\operatorname{loc}}(\R^n)$ as $j \to \infty$. The inequality \eqref{norm-cont} implies that $(\mathcal{F}^{-1}f \cdot \varphi_j)_{j \in \N}$ is a Cauchy sequence in  $E$. As $E$ is complete, we obtain that $ \mathcal{F}^{-1}f \cdot \varphi \in E$.
\end{proof}
\begin{remark}
	If $E$ is a DTMIB, Proposition \ref{fourier-solid} is a consequence of  \cite[Theorem 3]{D-P-V}.
\end{remark}

\subsection{Discrete spaces associated to TMIB and DTMIB}  
Let $E$ be a TMIB or DTMIB on $\R^n$ and fix a  lattice $\Lambda$ in $\R^n$ and a bounded open neighborhood $U$ of zero such that the family of sets $\{ \lambda + U \mid \lambda \in \Lambda \}$ is pairwise disjoint. We will use the following  notion introduced in  \cite[Definition 5.1]{D-P-GaborFrameCharGenModSp} (with $B  = 0$).
\begin{definition}\label{def-discrete}
		Let $\chi \in \D(U) \setminus \{0\}$. We define the \emph{discrete space associated to $E$ on $\Lambda$} as
			\[ E_{d}(\Lambda) = E_{d, \chi}(\Lambda) := \left\{ c = (c_{\lambda})_{\lambda \in \Lambda} \in \C^{\Lambda} \mid S_{\chi}(c) := \sum_{\lambda \in \Lambda} c_{\lambda} T_{\lambda} \chi \in E \right\} \]
		and endow it with the norm $\|c\|_{E_{d}(\Lambda)} = \|c\|_{E_{d, \chi}(\Lambda)} := \|S_{\chi}(c)\|_{E}$. For $\Lambda = \Z^{n}$ we simply write $E_{d} = E_{d}(\Z^{n})$.
\end{definition}		
The space $E_{d}(\Lambda)$ is a Banach space whose definition is independent of the window $\chi \in \D(U) \setminus \{0\}$ and different non-zero elements of $\D(U)$ induce equivalent norms \cite[Theorem 5.2]{D-P-GaborFrameCharGenModSp}. In particular, $E_{d}(\Lambda)$ is independent of the chosen neighborhood $U$. By \cite[Proposition 5.5]{D-P-GaborFrameCharGenModSp}, the following continuous inclusions hold
	\begin{equation}
		\label{eq:ContIncEd}
		s(\Lambda) \rightarrow E_{d}(\Lambda) \rightarrow s^{\prime}(\Lambda) .
	\end{equation}
We now give some examples of discrete spaces associated to TMIB and DTMIB; see also \cite[Section 5]{D-P-GaborFrameCharGenModSp}.
	
\begin{examples}\label{examples-seq}
$(i)$ Suppose that $E$ is solid. Then, it holds that
$$
E_d(\Lambda) = \left\{  c  \in \C^\Lambda  \, \Big| \,\sum_{\lambda \in \Lambda} c_\lambda  1_{\lambda+U} \in E \right\},
$$
where $1_{A}$ denotes the characteristic function of a set $A \subseteq \R^n$, and the topology of $E_d(\Lambda)$ is induced by the norm 
$
\| c \| =  \| \sum_{\lambda \in \Lambda} |c_\lambda|  1_{\lambda+U} \|_E.
$
Hence, in the solid case, our definition coincides with the standard one used in the theory of atomic decompositions of coorbit spaces  \cite[Definition 3.4]{F-G-BSpIntegrGroupRepAtomicDecomp}. For a polynomially bounded weight function $w$ on $\R^n$ it holds that $(L^p_w)_d(\Lambda) = \ell^p_w(\Lambda)$, $1 \leq p \leq \infty$, and $(C_{0,w})_d(\Lambda)=c_{0,w}(\Lambda)$. Let $\Lambda_1, \Lambda_2$ be two lattices in $\R^n$. Then, $(L^{p_1}(L^{p_2}))_d(\Lambda_1 \times \Lambda_2) = \ell^{p_1}(\Lambda_1;\ell^{p_2}(\Lambda_2))$,  $1 \leq p_1,p_2 < \infty$.


$(ii)$ Let $E$ be solid.  We wish to determine  $(\mathcal{F}E)_d(\Lambda)$. 
We define the dual lattice of $\Lambda = A_\Lambda \Z^{n}$ as $\Lambda^\perp = (A_\Lambda^t)^{-1} \Z^n =  \{ \mu \in \R^n \, | \, \lambda \cdot \mu \in \Z, \, \forall \lambda \in \Lambda \}$ and set $\operatorname*{vol}(\Lambda) = | A_\Lambda[0,1)^n |$.
We write $E(\R^n/ \Lambda^\perp)$ for the Banach space consisting of all $\Lambda^\perp$-periodic elements $f \in  L^1_{\operatorname{loc}}(\R^n)$   such that $\| f \|_{E(\R^n/ \Lambda^\perp)} := \| f  1_{I_{\Lambda^\perp}}\|_E < \infty$. Note that $E(\R^n/ \Lambda^\perp) \subset L^1(\R^n/ \Lambda^\perp)$. As customary, we define the Fourier coefficients of an element $f \in  L^1(\R^n/ \Lambda^\perp)$ as
$$
c_\lambda(f) =  \frac{1}{\operatorname*{vol}(\Lambda^\perp)} \int_{I_{\Lambda^\perp}} f(x) e^{-2\pi i \lambda \cdot x} dx, \qquad \lambda \in \Lambda.
$$
In \cite[Proposition 5.14]{D-P-GaborFrameCharGenModSp} it is shown that the mapping
$$
E(\R^n/ \Lambda^\perp) \rightarrow (\mathcal{F}E)_d(\Lambda), \, f \mapsto (c_\lambda(f))_{\lambda \in \Lambda}
$$
is a topological isomorphism.

$(iii)$ Let $E_j$ be a TMIB on $\R^{n_j}$ and let $\Lambda_j$ be a lattice in $\R^{n_j}$ for $j = 1,2$. Let $\tau$ denote either $\pi$ or $\varepsilon$. By \cite[Proposition 5.25]{D-P-GaborFrameCharGenModSp}, we have the canonical isomorphism
$(E_1 \widehat{\otimes}_{\tau} E_2)_d(\Lambda_1 \times \Lambda_2) \cong (E_1)_d(\Lambda_1) \widehat{\otimes}_{\tau} (E_2)_d(\Lambda_2)$.
\end{examples}	
Finally, we state two results from  \cite[Section 5]{D-P-GaborFrameCharGenModSp} that will play a crucial role in this article. Note that for all $c \in s^{\prime}(\Lambda)$ and $\varphi \in \S(\R^n)$ it holds that
$$
S_{\varphi}(c) := \sum_{\lambda \in \Lambda} c_{\lambda} T_{\lambda} \varphi \in \S'(\R^n)
$$
and that the series is absolutely summable in  $\S'(\R^n)$. We then have:
\begin{proposition}\label{S-result} \cite[Proposition 5.7]{D-P-GaborFrameCharGenModSp} 
The bilinear mapping
\[ E_{d}(\Lambda) \times \S(\R^{n}) \rightarrow E, \, (c, \varphi) \mapsto S_{\varphi}(c) \]
is well-defined and continuous.
\end{proposition}
\begin{remark}
We emphasize that Proposition \ref{S-result} states nothing about how  the series $S_{\varphi}(c) = \sum_{\lambda \in \Lambda} c_{\lambda} T_{\lambda} \varphi$ converges in $E$. We refer to \cite[Section 5.2]{D-P-GaborFrameCharGenModSp} for a discussion of this matter. 
\end{remark}
\begin{proposition}\label{eq:SpsiIntoE} \cite[Proposition 5.10]{D-P-GaborFrameCharGenModSp} 
The bilinear mapping
\[ E \times \S(\R^{n}) \rightarrow E_{d}(\Lambda), \, (e, \varphi) \mapsto R_{\varphi}(e) := (e * \varphi(\lambda))_{\lambda \in \Lambda}  \]
is well-defined and continuous.
\end{proposition}

\section{Statement of the main result and some consequences}\label{sect-stat}
We fix throughout this section a TMIB or DTMIB $E$ on $\R^{n}$. The main result of this article is now  as follows.
\begin{theorem}\label{main}
The following topological isomorphisms hold
\[ \D_{E} \cong E_{d} \widehat{\otimes} s \quad \text{and} \quad \D^{\prime}_{E} \cong E_{d} \widehat{\otimes} s^{\prime} . \]
\end{theorem}		
The proof of Theorem \ref{main} will be given at the end of the next section. Here,  we will  use Theorem \ref{main} to obtain some linear topological properties of the spaces $\D_{E}$  and $\D'_E$.

\begin{proposition}\label{top} \mbox{}
			\begin{itemize}
				\item[$(i)$]  The space $\D_{E}$ is a quasinormable Fr\'{e}chet space.
				\item[$(ii)$] The space $\D'_{E}$ is a complete $(LB)$-space and, thus, ultrabornological.
				\item[$(iii)$] Both the spaces $\D_{E}$ and $\D'_{E}$ are not Montel.
			\end{itemize}
\end{proposition}
\begin{proof}
$(i)$ As $s$ is quasinormable, 
 it follows from Theorem \ref{main} and \cite[Chapitre II, p.\ 76, Proposition 13]{G-ProdTensTopEspNucl} that $\D_{E}$ is quasinormable if and only if $E_{d}$ is quasinormable. The latter is true for any Banach space.

$(ii)$ Theorem \ref{main} implies that $\D^{\prime}_{E}$ is complete, while it is an $(LB)$-space by \eqref{tensor-2}. 

$(iii)$ As both $s$ and  $s'$ are Montel, it follows from Theorem \ref{main} and \cite[Chapitre II, p.\ 76, Proposition 13]{G-ProdTensTopEspNucl} that $\D_{E}$ and $\D^{\prime}_{E}$ are Montel if and only if $E_{d}$ is Montel. The latter is never true for an infinite dimensional Banach space.
\end{proof}
\begin{remark}
Proposition \ref{top}$(i)$ was shown in \cite[Theorem 5.6]{D-quasi} for the more general class of translation invariant  Banach spaces of distributions (see also \cite[Remark 7]{D-P-V}). 
 Proposition \ref{top}$(iii)$ generalizes \cite[Theorem 5.9]{D-quasi}, where this result is shown for solid TMIB, and gives a negative answer to the question asked in \cite[Remark 6]{D-P-V} for the class of TMIB. 
\end{remark}

\begin{corollary}\label{cor} Suppose that $E$ is a TMIB. Then, $(\mathcal{D}_E)' = \D'_{E'}$ as locally convex spaces. 
\end{corollary}
\begin{proof}
As already stated in Remark \ref{dual}, it was shown in \cite{D-P-V} that $(\mathcal{D}_E)' = \D'_{E'}$ as vector spaces and that the inclusion mapping $(\mathcal{D}_E)' \rightarrow \D'_{E'}$ is continuous. Since $(\mathcal{D}_E)'$ is webbed and $\D'_{E'}$ is ultrabornological (Proposition \ref{top}$(ii)$), the result is a consequence of De Wilde's open mapping theorem  \cite[Theorem 24.30]{M-V}.
\end{proof}
\begin{remark}
Corollary \ref{cor} implies that the equalities in \eqref{dualexc} hold topologically: the second and third one are particular instances of this result, while the last one is a consequence of the third one. The first one follows by noting that $\D'_{(C_{0,1/w})'} =  \D'_{L^1_w}$ as locally convex spaces or can be shown in the same way as Corollary \ref{cor} (the inclusion mapping  $(\dot{\mathcal{B}}_{1/w})' \rightarrow \D'_{L^1_w}$ is continuous by \cite[Corollary 5]{D-P-V}). 	
\end{remark}	
\begin{corollary}\label{cor-2} Suppose that $E$ is solid. Then, $\mathcal{D}'_{\mathcal{F}E} =\mathcal{F}E_{-\infty}$ as locally convex spaces.
\end{corollary}
\begin{proof}
By Propsition \ref{fourier-solid}, $\mathcal{D}'_{\mathcal{F}E} =\mathcal{F}E_{-\infty}$ as vector spaces. Moreover, it is clear that the inclusion mapping $\mathcal{F}E_{-\infty}  \to 	\mathcal{D}'_{\mathcal{F}E}$ is continuous. As $\mathcal{F}E_{-\infty}$ is webbed and $	\mathcal{D}'_{\mathcal{F}E}$ is ultrabornological (Proposition \ref{top}$(ii)$), the result is again a consequence of De Wilde's open mapping theorem \cite[Theorem 24.30]{M-V}.
\end{proof}

Finally, we  combine Theorem \ref{main} with Example \ref{examples-seq} to recover some well-known sequence space representations as well as obtain several new ones. 

\begin{examples}\label{examples-main}
$(i)$ Let $w$ be a polynomially bounded weight function on $\R^n$. We have that
$$
\D_{L^p_w} \cong \ell^p_w  \widehat{\otimes} s   \quad (1 \leq p < \infty), \qquad \mathcal{B}_{w} \cong \ell^\infty_w  \widehat{\otimes} s, \qquad  \dot{\mathcal{B}}_w  \cong c_{0,w}  \widehat{\otimes} s ,
$$
and
$$
\D'_{L^p_w} \cong \ell^p_w  \widehat{\otimes} s'   \quad (1 \leq p < \infty), \qquad \mathcal{B}'_{w} \cong \ell^\infty_w  \widehat{\otimes} s', \qquad  \dot{\mathcal{B}}_w \cong c_{0,w}  \widehat{\otimes} s' .
$$

$(ii)$ For $1 \leq p_1, p_2 < \infty$  we have that
$$
\D_{L^{p_1}(L^{p_2})} \cong \ell^{p_1}(\ell^{p_2})  \widehat{\otimes} s \qquad \mbox{and} \qquad \D'_{L^{p_1}(L^{p_2})} \cong \ell^{p_1}(\ell^{p_2})   \widehat{\otimes} s'. 
$$

$(iii)$ Suppose that $E$ is solid. By Example \ref{examples-test}$(ii)$  and Corollary \ref{cor-2}, we have that 
$$
\mathcal{F}E_{\infty} \cong (\mathcal{F}E)_d  \widehat{\otimes} s \qquad \mbox{and} \qquad \mathcal{F}E_{-\infty} \cong (\mathcal{F}E)_d  \widehat{\otimes} s^{\prime}. 
$$
We refer to Example \ref{examples-seq}$(ii)$  for an explicit description of the discrete space $(\mathcal{F}E)_d$ in terms of Fourier coefficients. For $E = L^p$, $1 \leq p \leq \infty$, this proves the isomorphisms in \eqref{sobolev-seq}. 

$(iv)$  Let $E_j$ be a TMIB on $\R^{n_j}$ for $j = 1,2$. Let $\tau$ denote either $\pi$ or $\varepsilon$. By Example \ref{examples-seq}$(iii)$, it holds that
$$
\mathcal{D}_{E_1 \widehat{\otimes}_\tau E_2} \cong (E_1)_d \widehat{\otimes}_\tau (E_2)_d  \widehat{\otimes}_\tau s, \qquad \mathcal{D}'_{E_1 \widehat{\otimes}_\tau E_2} \cong (E_1)_d \widehat{\otimes}_\tau (E_2)_d  \widehat{\otimes}_\tau s'.
$$
In particular, we have that for $1 \leq p_1, p_2 < \infty$
$$
\mathcal{D}_{L^{p_1} \widehat{\otimes}_\tau L^{p_2}} \cong \ell^{p_1} \widehat{\otimes}_\tau  \ell^{p_2} \widehat{\otimes}_\tau s, \qquad \mathcal{D}'_{L^{p_1} \widehat{\otimes}_\tau L^{p_2}} \cong \ell^{p_1} \widehat{\otimes}_\tau \ell^{p_2}\widehat{\otimes}_\tau s'.
$$
\end{examples}

\section{The Gabor frame characterization of $\D_{E}$ and $\D^{\prime}_{E}$}
\label{sec:GaborFrameChar}

\subsection{The short-time Fourier transform and frame operators on $\S^{\prime}(\R^{n})$}\label{subsect:info}

We briefly discuss the short-time Fourier transform and Gabor frames on $\S^{\prime}(\R^{n})$, we refer to \cite{G-FoundationsTimeFreqAnal} for more information. For $z = (x, \xi) \in \R^{2n}$, we write $\pi(z) = \pi(x, \xi) = M_{\xi} T_{x}$ for a time-frequency shift. We define the \emph{short-time Fourier transform} (STFT) of $f \in \S'(\R^{n})$ with respect to a window $\psi \in \S(\R^{n})$ as
	\[ V_{\psi} f(z) = V_{\psi} f(x, \xi) := \ev{f}{\pi(x, -\xi) \overline{\psi}}, \qquad z = (x, \xi) \in \R^{2n} . \] 
Then, $V_{\psi} f \in C(\R^{2n})$ and $\| V_{\psi} f\|_{L^{\infty}_{(1 + |\cdot|)^{-M}}} < \infty$ for some $M \in \N$. If $A \subset \S(\R^{n})$ is a bounded set,  the previous estimate holds uniformly for $f \in A$. As $\S^{\prime}(\R^{n})$ is bornological, it follows that $V_{\psi} : \S^{\prime}(\R^{n}) \rightarrow \S^{\prime}(\R^{2n})$ is continuous. Furthermore, the STFT induces the continuous mapping $V_{\psi} : \S(\R^{n}) \rightarrow \S(\R^{2n})$.

Next we consider Gabor frames. Fix a lattice $\Lambda$ in $\R^{2n}$. Given a window $\psi \in \S(\R^{n})$, we consider the \emph{analysis operator}
	\[ C_{\psi}  : \S^{\prime}(\R^{n}) \rightarrow s^{\prime}(\Lambda), \, \quad f \mapsto (V_{\psi} f(\lambda))_{\lambda \in \Lambda} , \]
and the \emph{synthesis operator}
	\[ D_{\psi} : s^{\prime}(\Lambda) \rightarrow \S^{\prime}(\R^{n}), \, \quad c \mapsto \sum_{\lambda \in \Lambda} c_{\lambda} \pi(\lambda) \psi , \]
which are both continuous linear mappings. Moreover, the series $\sum_{\lambda \in \Lambda} c_{\lambda} \pi(\lambda) \psi$ is absolutely summable in $\S^{\prime}(\R^{n})$. For $\psi, \gamma \in \S(\R^{n})$ we define
	\[ S_{\psi, \gamma} := D_{\gamma} \circ C_{\psi} : \S^{\prime}(\R^{n}) \rightarrow \S^{\prime}(\R^{n}) . \]
If $S_{\psi, \gamma} = \id_{\S^{\prime}(\R^{n})}$ we call $(\psi, \gamma)$ a \emph{pair of dual windows} on $\Lambda$. On lattices of the type $a \Z^{n} \times b \Z^{n}$, with $a, b > 0$, pairs of dual windows may be characterized by the Wexler-Raz biorthogonality relations.

	\begin{theorem}[{\cite[Theorem 7.3.1 and the subsequent remark]{G-FoundationsTimeFreqAnal}}]
		\label{eq:WexlerRaz}
		Let $\psi, \gamma \in \S(\R^{n})$ and let $a, b > 0$. Then, $(\psi, \gamma)$ is a pair of dual windows on $a\Z^{n} \times b\Z^{n}$ if and only if
			$$
				\int_{\R^n}\pi(k, l) \psi(x) \overline{\pi(k', l') \gamma(x)} dx = (ab)^{n} \delta_{k, k'} \delta_{l, l'} , \qquad (k, l), (k', l') \in \frac{1}{a} \Z^{n} \times \frac{1}{b} \Z^{n} ,
			$$
		or equivalently
			$$
				\frac{1}{(ab)^{n}} C_{\psi, \frac{1}{a} \Z^{n} \times \frac{1}{b} \Z^{n}} \circ D_{\gamma, \frac{1}{a} \Z^{n} \times \frac{1}{b} \Z^{n}} = \id_{s^{\prime}(\frac{1}{a} \Z^{n} \times \frac{1}{b} \Z^{n})} .
			$$
	\end{theorem}
	
Given a window $\psi \in \S(\R^{n})$, the mapping $C_{\psi} : L^{2}(\R^{n}) \rightarrow \ell^{2}(\Lambda)$ and its adjoint $D_{\psi} :  \ell^{2}(\Lambda) \rightarrow L^{2}(\R^{n})$ are   well-defined and continuous. The set of time-frequency shifts
	\[ \mathcal{G}(\Lambda, \psi) := \{ \pi(\lambda) \psi \mid \lambda \in \Lambda \} \]
is called a \emph{Gabor frame} if there are $A, B > 0$ such that
	\[ A \|f\|_{L^{2}} \leq \|(V_{\psi}f(\lambda))_{\lambda \in \Lambda}\|_{l^{2}(\Lambda)} \leq B \|f\|_{L^{2}} , \quad f \in L^{2}(\R^{n}) . \]
In this case $S = S_{\psi, \psi}$ is a bounded positive invertible linear operator on $L^{2}(\R^{n})$. Set $\gamma^{\circ} = S^{-1} \psi$. Suppose that $\Lambda  = a \Z^{n} \times b \Z^{n}$ for some $a,b > 0$. A classical result of Janssen \cite[Corollary 13.5.4]{G-FoundationsTimeFreqAnal} states that $\gamma^{\circ} \in \S(\R^{n})$. As $S$ and $\pi$ commute on $\Lambda$, $(\psi, \gamma^{\circ})$ form a pair of dual windows on $\Lambda$.  We call $\gamma^\circ$ the \emph{canonical dual window} of $\psi$. Finally, we remark that, for the Gaussian $\psi(x) = e^{-\pi|x|^2}$, $\mathcal{G}(a\Z^{n} \times b\Z^{n}, \psi)$ is a Gabor frame if and only if $ab < 1$ (cf.\  \cite[Theorem 7.5.3]{G-FoundationsTimeFreqAnal}).


\subsection{Continuity of the frame operators on $\D_{E}$ and $\D^{\prime}_{E}$}

Let $E$ be a TMIB or DTMIB on $\R^n$. We fix two lattices $\Lambda_{0}$ and $\Lambda_{1}$ in $\R^{n}$ and consider the product lattice $\Lambda = \Lambda_{0} \times \Lambda_{1}$ in $\R^{2n}$. We now establish the mapping properties of the analysis and synthesis operators on the spaces $\D_{E}$ and $\D^{\prime}_{E}$. 
	\begin{proposition}
		\label{p:GaborDE}
		For $\psi \in \S(\R^{n})$ the mappings
			\[ C_{\psi} : \D_{E} \rightarrow s(\Lambda_{1}; E_{d}(\Lambda_{0})) \qquad \mbox{and} \qquad D_{\psi} : s(\Lambda_{1}; E_{d}(\Lambda_{0})) \rightarrow \D_{E} \]
		are well-defined and continuous. 
	\end{proposition}
	
	\begin{proof} Throughout the proof $\tau_{j},C_j$, $j = 0,1$, will denote the constants occurring in \eqref{twf}. Furthermore, we will make use of the mappings from Proposition \ref{S-result} and Proposition \ref{eq:SpsiIntoE}.
	
		We first consider $C_{\psi}$.  Note that $V_{\psi} g(x, \xi)  = (M_{-\xi} g) * \check{\overline{\psi}} (x)$ for $g \in \S'(\R^n)$. Hence,  for $\lambda_1 \in \Lambda_1$ fixed, it holds that 
		 $$
		 (V_{\varphi} g(\lambda_{0}, \lambda_{1}))_{\lambda_{0}\in \Lambda_0} = ((M_{-\lambda_1} g) \ast \check{\overline{\varphi}}(\lambda_0))_{\lambda_0 \in \Lambda_0} = R_{\check{\overline{\varphi}}}(M_{-\lambda_1} g)
		 $$
		 for all $g \in \S'(\R^n)$ and $\varphi \in \S(\R^n)$. 
		 By Proposition \ref{eq:SpsiIntoE}, we have that for all $e \in E$ and  $\varphi \in \S(\R^n)$ it holds that $(V_{\varphi} e(\lambda_{0}, \lambda_{1}))_{\lambda_{0}\in \Lambda_0} \in E_d(\Lambda_0)$ and that there are $C > 0$ and $N \in \N$ (independent of $e$ and $\varphi$) such that 
$$
 \| (V_{\varphi} e(\lambda_{0}, \lambda_{1}))_{\lambda_{0} \in \Lambda_{0}} \|_{E_{d}(\Lambda_{0})} \leq C \|M_{-\lambda_{1}} e\|_{E} \|\varphi\|_{\S^{N}} \leq C C_{1} (1 + |\lambda_{1}|)^{\tau_{1}} \|e\|_{E} \|\varphi\|_{\S^{N}}
$$
  For all $f \in \S^{\prime}(\R^{n})$ and $\alpha \in \N^{n}$  it holds that 
			\[ (2 \pi i \xi)^{\alpha} V_{\psi} f(x, \xi) = \sum_{\beta \leq \alpha} (-1)^{|\beta|} {\alpha \choose \beta} V_{\psi^{(\alpha - \beta)}} f^{(\beta)}(x, \xi) . \]
We obtain that, for all $f \in \D_E$,
			\[ \| (V_{\psi} f(\lambda_{0}, \lambda_{1}))_{\lambda_{0} \in \Lambda_{0}} \|_{E_{d}(\Lambda_{0})} \leq C C_{1} (\sqrt{n} / \pi)^{|\alpha|} \|\psi\|_{\S^{N + |\alpha|}} \frac{(1 + |\lambda_{1}|)^{\tau_{1}}}{|\lambda_{1}|^{|\alpha|}} \|f\|_{E, |\alpha|} . \]
	This shows that $C_{\psi} : \D_{E} \rightarrow s(\Lambda_{1}; E_{d}(\Lambda_{0}))$ is well-defined and continuous.
		
		Next, we consider $D_{\psi}$. For all $c = (c_{\lambda_0,\lambda_1})_{(\lambda_0,\lambda_1) \in \Lambda} \in s'(\Lambda)$ and $\alpha \in \N^n$ it holds that
		\begin{align*}
		\partial^{\alpha} D_{\psi}(c) &= \sum_{\beta \leq \alpha} {\alpha \choose \beta} \sum_{\lambda_{1} \in \Lambda_{1}} (2 \pi i \lambda_{1})^{\beta} M_{\lambda_{1}} \sum_{\lambda_{0} \in \Lambda_{0}} c_{\lambda_0, \lambda_1} T_{\lambda_{0}} \psi^{(\alpha - \beta)} \\
		&= \sum_{\beta \leq \alpha} {\alpha \choose \beta} \sum_{\lambda_{1} \in \Lambda_{1}} (2 \pi i \lambda_{1})^{\beta} M_{\lambda_{1}} S_{\psi^{(\alpha - \beta)}} ((c_{\lambda_0,\lambda_1})_{\lambda_0 \in \Lambda_0}) .
		\end{align*}
Hence, Proposition \ref{S-result}  implies that for all  $c  \in s(\Lambda_{1}; E_{d}(\Lambda_{0}))$ and $\alpha \in \N^n$ it holds that $\partial^{\alpha} D_{\psi}(c) \in E$ and that there are $C >0$ and $N \in \N$ (independent of $c$ and $\alpha$) such that
	\begin{align*}
	\|\partial^{\alpha} D_{\psi}(c)\|_E &\leq C_1 (4\pi)^{|\alpha|} \sum_{\lambda_{1} \in \Lambda_{1}} (1 +|\lambda_1|)^{\tau_1 + |\alpha|} \max_{\beta \leq \alpha} \|S_{\psi^{(\alpha - \beta)}} ((c_{\lambda_0,\lambda_1})_{\lambda_0 \in \Lambda_0})\|_E \\
	&\leq CC_1 (4\pi)^{|\alpha|}\|\psi\|_{\S^{N + |\alpha|}} \|c\|_{\ell^{1}_{(1 + |\cdot|)^{\tau_{1} + |\alpha|}}(\Lambda_{1}; E_{d}(\Lambda_{0}))} .
\end{align*}
This shows that $D_{\psi} : s(\Lambda_{1}; E_{d}(\Lambda_{0})) \rightarrow \D_{E}$ is well-defined and continuous.
\end{proof}
	
	

	
%

	\begin{proposition}
		\label{p:GaborD'E}
		For  $\psi \in \S(\R^{n})$ the mappings
			\[ C_{\psi} : \D^{\prime}_{E} \rightarrow s^{\prime}(\Lambda_{1}; E_{d}(\Lambda_{0})) \qquad \mbox{and} \qquad D_{\psi} : s^{\prime}(\Lambda_{1}; E_{d}(\Lambda_{0})) \rightarrow \D^{\prime}_{E} \]
		are well-defined and continuous. 
	\end{proposition}
	
	\begin{proof}
		 We first consider $C_\psi$. By  \cite[Lemma 1]{M-DistElemSpDist}, there are $\psi_{0}, \psi_{1} \in \S(\R^{n})$ such that $\psi = \psi_{0} * \psi_{1}$. Consider the mappings 
			\begin{align*}
				A_{0} &: \D^{\prime}_{E} \rightarrow \mathcal{L}_{b}(\S(\R^{n}), E) , f \mapsto (\varphi \mapsto f \ast \varphi), \\ 
				A_{1} &: \mathcal{L}_{b}(\S(\R^{n}), E) \rightarrow s^{\prime}( \Lambda_{1};E_{d}(\Lambda_{0})) , T \mapsto ((T(M_{\lambda_{1}} \check{\overline{\psi_{0}}}) * M_{\lambda_{1}} \check{\overline{\psi_{1}}})(\lambda_{0}))_{(\lambda_0, \lambda_1) \in \Lambda} , \\ 							
				A_{2} &: s^{\prime}(\Lambda_{1}; E_{d}(\Lambda_{0})) \rightarrow s^{\prime}( \Lambda_{1};E_{d}(\Lambda_{0})) , (c_{\lambda_0,\lambda_1})_{(\lambda_0, \lambda_1) \in \Lambda} \mapsto (e^{-2 \pi i \lambda_{0} \cdot \lambda_{1}} c_{\lambda_0,\lambda_1})_{(\lambda_0, \lambda_1) \in \Lambda}.
							\end{align*}  
		As $V_{\psi} f(x, \xi) = e^{-2 \pi i \xi \cdot x} (f * M_{\xi} \check{\overline{\psi_{0}}} * M_{\xi} \check{\overline{\psi_{1}}}) (x)$, it follows that $C_{\psi} = A_{2} \circ A_{1} \circ A_{0}$. Consequently, it suffices to verify that $A_0$, $A_1$, and $A_2$ are well-defined and continuous. For $A_{0}$ this holds true by definition of the topology of $\D^{\prime}_{E}$. Next, we consider $A_{1}$. 
		Since $\mathcal{S}(\R^n) \cong s$ is a nuclear Fr\'echet space, we have that
$$
\mathcal{L}_{b}(\S(\R^{n}), E)  \cong \mathcal{S}'(\R^n) \widehat{\otimes} E \cong s' \widehat{\otimes}  E.
$$		
Hence, \eqref{tensor-2} implies that $\mathcal{L}_{b}(\S(\R^{n}), E)$ is an $(LB)$-space and thus bornological. Therefore, it is enough to show that $A_{1}$ maps bounded sets into bounded sets. Let $B$ be an arbitrary bounded subset of $\mathcal{L}_{b}(\S(\R^{n}), E)$. As the set $B$ is equicontinuous by the Banach-Steinhaus theorem, there are $C > 0$ and $N \in \N$ such that
\[ \sup_{T \in B} \|T(M_{\lambda_{1}} \check{\overline{\psi_{0}}})\|_{E} \leq C \|M_{\lambda_{1}} \check{\overline{\psi_{0}}}\|_{\S^{N}} \leq C (4 \pi)^{N} \|\check{\overline{\psi_{0}}}\|_{\S^{N}} (1 + |\lambda_{1}|)^{N}, \qquad \lambda_1 \in \Lambda_1.  \]
Proposition \ref{eq:SpsiIntoE}  yields that, for fixed $\lambda_1 \in \Lambda_1$,  $((T(M_{\lambda_{1}} \check{\overline{\psi_{0}}}) * M_{\lambda_{1}} \check{\overline{\psi_{1}}})(\lambda_{0}))_{\lambda_0 \in \Lambda_0} = R_{M_{\lambda_1} \check{\overline{\psi_{1}}}}(T(M_{\lambda_1}\check{\overline{\psi_{0}}})) \in E_d(\Lambda_0)$ for all $T \in B$ and that there are $C' > 0$ and $N' \in \N$ such that 
\begin{align*}
&\sup_{T \in B} \|((T(M_{\lambda_{1}} \check{\overline{\psi_{0}}}) * M_{\lambda_{1}} \check{\overline{\psi_{1}}})(\lambda_{0}))_{\lambda_{0} \in \Lambda_{0}}\|_{E_{d}(\Lambda_{0})} \\
&\qquad \qquad \qquad \qquad \leq C'\sup_{T \in B} \|T(M_{\lambda_1}\check{\overline{\psi_{0}}})\|_E \|M_{\lambda_1} \check{\overline{\psi_{1}}}\|_{\S^{N'}}\\
&\qquad \qquad \qquad \qquad \leq CC' (4 \pi)^{N+N'} \|\check{\overline{\psi_{0}}}\|_{\S^{N}} \| \check{\overline{\psi_{1}}}\|_{\S^{N'}}(1 + |\lambda_{1}|)^{N+N'}, 
\end{align*}
which shows that $A_1(B)$ is bounded  in $s^{\prime}( \Lambda_{1};E_{d}(\Lambda_{0}))$. Finally, we consider $A_{2}$.  Fix $\chi \in \D(U) \backslash \{0\}$.  For all $c = (c_{\lambda_0,\lambda_1})_{(\lambda_0,\lambda_1) \in \Lambda} \in s'(\Lambda)$ it holds that
$$
S_{\chi}( (e^{-2 \pi i \lambda_{0} \cdot \lambda_{1}} c_{\lambda_0,\lambda_1})_{\lambda_0 \in \Lambda_0}) = M_{\lambda_1}S_{M_{\lambda_1}\chi}( (c_{\lambda_0,\lambda_1})_{\lambda_0 \in \Lambda_0}), \qquad \lambda_1 \in \Lambda_1, 
$$
whence the statement follows from Proposition \ref{S-result}.

Next, we consider $D_{\psi}$. For all $c = (c_{\lambda_0,\lambda_1})_{(\lambda_0,\lambda_1) \in \Lambda} \in s'(\Lambda)$ and $\varphi \in \mathcal{S}(\R^n)$ it holds that
		\begin{align*}
		 D_{\psi}(c) \ast \varphi &=  \sum_{(\lambda_0,\lambda_1) \in \Lambda}   c_{\lambda_0,\lambda_1} M_{\lambda_1} T_{\lambda_0} \psi \ast \varphi \\
		&= \sum_{(\lambda_0,\lambda_{1}) \in \Lambda}  c_{\lambda_0,\lambda_1} M_{\lambda_1} T_{\lambda_0} V_{\check{\overline{\psi}}} \varphi( \, \cdot \,, \lambda_1) \\
		&= \sum_{\lambda_{1}\in \Lambda_1} M_{\lambda_1} \sum_{\lambda_{0}\in \Lambda_0}   c_{\lambda_0,\lambda_1} T_{\lambda_0} V_{\check{\overline{\psi}}} \varphi( \, \cdot \,, \lambda_1) \\
		&=  \sum_{\lambda_{1}\in \Lambda_1} M_{\lambda_1}  S_{V_{\check{\overline{\psi}}} \varphi( \, \cdot \,, \lambda_1)} ((c_{\lambda_0,\lambda_1})_{\lambda_0 \in \Lambda_0})
		\end{align*}
The result is therefore a consequence of Proposition \ref{S-result} and the continuity of the mapping $V_{\check{\overline{\psi}}}: \S(\R^n) \rightarrow \S(\R^{2n})$.

	\end{proof}
	
We are now sufficiently prepared to prove our main result. 
	
\begin{proof}[Proof of Theorem \ref{main}]	
By the Pe\l czy\'nski decomposition method (in the version of Vogt) \cite[Proposition 1.2 and the subsequent remark]{V-SeqSpRepSpTestFuncDist}, it suffices to show that $\D_{E}$ and $\D'_E$ are isomorphic to a complemented subspace of  $E_{d} \widehat{\otimes} s$ and $E_{d} \widehat{\otimes} s^{\prime}$, respectively, and, vice versa, that $E_{d} \widehat{\otimes} s$ and $E_{d} \widehat{\otimes} s^{\prime}$ are isomorphic to a complemented subspace of $\D_{E}$ and $\D'_E$, respectively.  Fix $0 < b < 1$. By \eqref{tensor-1} and \eqref{tensor-2}, we have that  $E_{d} \widehat{\otimes} s \cong s(b\Z^n; E_{d}(\Z^n))$ and $E_{d} \widehat{\otimes} s' \cong s'(b\Z^n; E_{d}(\Z^n))$.  Set $\psi = e^{-\pi |x|^2} \in \S(\R^n)$. Then, $\mathcal{G}(\psi, \Z^n \times b\Z^n)$ is a Gabor frame and the canonical dual window $\gamma^\circ$ of $\psi$ also belongs to $\S(\R^n)$ (see the end of  Subsection \ref{subsect:info}). Since $(\psi, \gamma^\circ)$ is a pair of dual windows on $\Z^n \times b\Z^n$, we have that 
$$
D_{\gamma^\circ, \Z^{n} \times b \Z^{n}} \circ C_{\psi, \Z^{n} \times b \Z^{n}} = \id_{\S'(\R^n)}
$$
and (Proposition \ref{eq:WexlerRaz})
$$
\frac{1}{b^{n}} C_{\psi, \Z^{n} \times \frac{1}{b} \Z^{n}} \circ D_{\gamma^\circ, \Z^{n} \times \frac{1}{b} \Z^{n}} = \id_{s^{\prime}(\Z^{n} \times \frac{1}{b} \Z^{n})} .
$$
Hence, the statements follow from Proposition \ref{p:GaborDE} and Proposition \ref{p:GaborD'E}.

\end{proof}

\section{The isomorphic structure of the spaces $\D_{L^{p_1} \widehat{\otimes}_\pi L^{p_2}}$}\label{sect-iso}
In this final section we use the sequence space representations from Example \ref{examples-main} to study the isomorphic structure of the spaces $\D_{L^{p_1} \widehat{\otimes}_\pi L^{p_2}}$. Our goal is to show that these spaces are genuinely new and distinct in the following sense.
\begin{proposition}\label{isotensor-2} Let $1 <  p_1, p_2 < \infty$.   
\begin{itemize}
\item[$(i)$]  $\D_{L^{p_1} \widehat{\otimes}_\pi L^{p_2}}$ is not isomorphic to a complemented subspace of $\mathcal{D}_{L^{q_1}(L^{q_2})}$, $1 \leq q_1, q_2 < \infty$, $\dot{\mathcal{B}}$, or $\mathcal{B}$.
\item[$(ii)$] Let $1 < q_1, q_2 < \infty$.  Then, $\D_{L^{p_1} \widehat{\otimes}_\pi L^{p_2}} \cong \D_{L^{q_1} \widehat{\otimes}_\pi L^{q_2}}$ if and only if $\{p_1,p_2\} = \{q_1, q_2\}$.
\end{itemize}
\end{proposition}
We need some preparation for the proof of Proposition \ref{isotensor-2}. Given two locally convex spaces $X$ and $Y$, we write $Y \underset{c}{\subset} X$ to indicate that  $Y$ is isomorphic to a complemented subspace of $X$. We denote by $\operatorname{csn}(X)$ the family of all continuous seminorms on $X$. Furthermore, for $p \in \operatorname{csn}(X)$ we denote by $X_p$ the local Banach space associated to $p$, that is, the completion of $X / \ker p$ with respect to the induced norm $p$. We have the following simple but useful observation from \cite{Diaz}. 

\begin{lemma}\cite[Lemma 1]{Diaz} \label{key}
Let $X$ be a locally convex space. If $Y$ is a normed space with $Y \underset{c}{\subset}X$, then there is a $p_0 \in \operatorname{csn}(X)$ such that $Y \underset{c}{\subset} X_p$ for all $p \in \operatorname{csn}(X)$ with $p \geq p_0$. \end{lemma}

\begin{corollary}\label{basic}
Let $X$ and $Y$ be Banach spaces. Suppose that $\ell^p(X) \cong X$ for some $1 \leq p < \infty$ or $c_0(X) \cong X$. 
\begin{itemize} 
\item[$(i)$] If $s(Y) \underset{c}{\subset} s(X)$, then $Y \underset{c}{\subset} X$. 
\item[$(ii)$] Assume additionally that $\ell^p(Y) \cong Y$ for some $1 \leq p < \infty$ or $c_0(Y) \cong Y$.  If $s(X) \cong s(Y)$, then $X \cong Y$.
\end{itemize}
\end{corollary}
\begin{proof}
$(i)$ Suppose that $\ell^p(X) \cong X$ for some $1 \leq p < \infty$, the case $c_0(X) \cong X$ is similar. Since $Y \underset{c}{\subset} s(Y)$, we obtain that $Y \underset{c}{\subset} s(X)$.  As $\{ \| \, \cdot \, \|_{\ell^p_{(1+|\, \cdot \,|)^r}(X)} \, | \, r > 0 \}$ is a fundamental system of norms for $s(X)$ and the local Banach space associated to $\| \, \cdot \, \|_{\ell^p_{(1+|\, \cdot \,|)^r}(X)}$, $r > 0$, is given by $\ell^p_{(1+|\, \cdot \,|)^r}(X)$, Lemma \ref{key} implies that  there exists a $r > 0$ such that  $Y \underset{c}{\subset} \ell^p_{(1+|\, \cdot \,|)^r}(X) \cong \ell^p(X) \cong X$. 

\noindent $(ii)$ We have that both $X \underset{c}{\subset} Y$ and $Y \underset{c}{\subset} X$ by $(i)$. Hence, the result follows from the Pe\l czy\'nski  decomposition method \cite[Theorem 2.2.3]{AK}.
\end{proof}

Next, we recall several fundamental results concerning the isomorphic structure of the spaces  $\ell^{p_1} \widehat{\otimes}_\pi \ell^{p_2}$ from \cite{AF,GL}.  We set 
$$
t(p_1,p_2) = \frac{1}{\min \{ 1, \frac{1}{p_1} + \frac{1}{p_2}\}}, \qquad p_1,p_2 > 1.
$$

\begin{lemma} \label{cool}  Let $1 < p_1,p_2 < \infty$. 
\begin{itemize}
\item[$(i)$] \cite[Corollary 3.6]{GL} $\ell^{p_1} \widehat{\otimes}_\pi \ell^{p_2}$ is not isomorphic to a complemented subspace of a Banach space with an unconditional basis.
\item[$(ii)$] \cite[Theorem 1.3]{AF} $\ell^{t(p_1,p_2)}( \ell^{p_1} \widehat{\otimes}_\pi \ell^{p_2}) \cong  \ell^{p_1} \widehat{\otimes}_\pi \ell^{p_2}$.
\item[$(iii)$] \cite[Theorem 4.1]{AF} Let $1 \leq p < \infty$. Then, $\ell^p$ is isomorphic to a closed subspace of  $\ell^{p_1} \widehat{\otimes}_\pi \ell^{p_2}$ if and only if $p \in \{ p_1,p_2, t(p_1,p_2) \}$.
\end{itemize}
\end{lemma}

\begin{proof}[Proof of Proposition \ref{isotensor-2}]  We employ the sequence space representations from Example \ref{examples-main}.

$(i)$ Suppose first that $\D_{L^{p_1} \widehat{\otimes}_\pi L^{p_2}}  \underset{c}{\subset}  \mathcal{D}_{L^{q_1}(L^{q_2})}$ or $\dot{\mathcal{B}}$. Corollary \ref{basic}$(i)$ would then imply that $\ell^{p_1} \widehat{\otimes}_\pi \ell^{p_2} \underset{c}{\subset} \ell^{q_1}(\ell^{q_2})$ or $c_0$. As the latter spaces have an unconditional basis, this is impossible by Lemma \ref{cool}$(i)$. Next, suppose that $\D_{L^{p_1} \widehat{\otimes}_\pi L^{p_2}}  \underset{c}{\subset} \mathcal{B}$. Proceeding as in the proof of Corollary \ref{basic}$(i)$, we obtain that $\ell^{p_1} \widehat{\otimes}_\pi \ell^{p_2} \underset{c}{\subset}  \ell^1(\ell^\infty)$ and thus $\ell^{p_1},\ell^{p_2} \underset{c}{\subset}  \ell^1(\ell^\infty)$. As $p_1, p_2 \neq 1$, this is impossible by \cite[Theorem 3.2]{CM}. 

$(ii)$ By Corollary \ref{basic}$(ii)$ and Lemma \ref{cool}$(ii)$, we have that $\D_{L^{p_1} \widehat{\otimes}_\pi L^{p_2}}  \cong \D_{L^{q_1} \widehat{\otimes}_\pi L^{q_2}}$ if and only if $\ell^{p_1} \widehat{\otimes}_\pi \ell^{p_2}  \cong \ell^{q_1} \widehat{\otimes}_\pi \ell^{q_2}$. The result therefore follows from Lemma \ref{cool}$(iii)$.
\end{proof}

Finally, we show how Corollary \ref{basic} may also be used to discuss the isomorphic structure of  the spaces $\mathcal{D}_{L^p}$, $\dot{\mathcal{B}}$, $\mathcal{B}$ and  $\D_{L^{p_1}(L^{p_2})}$. To the best of our knowledge, this problem has not yet been considered in the literature. 

\begin{theorem} \mbox{}
\begin{itemize}
\item[$(i)$] None of the spaces $\mathcal{D}_{L^p}$, $1 \leq p < \infty$, $\dot{\mathcal{B}}$, $\mathcal{B}$ is isomorphic to a complemented subspace of  any other of these spaces. 
\item[$(ii)$] The spaces $\mathcal{D}_{L^{p_1}(L^{p_2})}$,  $1 \leq p_1,p_2 < \infty$, are mutually  non-isomorphic. 
\end{itemize}
\end{theorem}
\begin{proof}
We again employ the sequence space representations from Example \ref{examples-main}.

$(i)$ We first show that  $\mathcal{D}_{L^p}$, $\dot{\mathcal{B}}$, and $\mathcal{B}$ are not isomorphic to a complemented subspace of any of the spaces $\mathcal{D}_{L^q}$, $1 \leq q < \infty$, $\dot{\mathcal{B}}$ (except for the space itself). By Corollary \ref{basic}$(i)$, this follows from the corresponding result for  $\ell^p$, $c_0$, and $\ell^\infty$ (cf.\ \cite[Corollary 2.1.6]{AK}).
Next, we show that  $\mathcal{D}_{L^p}$ and $\dot{\mathcal{B}}$  are not isomorphic to a complemented subspace of $\mathcal{B}$.  Suppose that $s(\ell^p)  \underset{c}{\subset} s(\ell^\infty)$.  Choose $1 \leq q < \infty$ with $p \neq q$. Proceeding as in the proof of Corollary \ref{basic}$(i)$, we obtain that $\ell^p   \underset{c}{\subset} \ell^q(\ell^\infty)$, which is false by \cite[Theorem 3.2]{CM}. Next, suppose that $s(c_0)  \underset{c}{\subset} s(\ell^\infty)$. As before, we may deduce that  $c_0  \underset{c}{\subset} \ell^1(\ell^\infty)$, which is false by the remark after \cite[Theorem 4.7]{BCLT}. 

$(ii)$ This follows from Corollary \ref{basic}$(ii)$ and the fact that the spaces  $\ell^{p_1}(\ell^{p_2})$, $1 \leq p_1,p_2 < \infty$, are mutually non-isomorphic \cite[Theorem 1.1]{CM}.
\end{proof}

\end{document}